\documentclass[oneside,11pt]{amsart}
\usepackage{amsmath, amsfonts,amsthm,times,graphics}

 \makeatletter
\renewcommand*\subjclass[2][2000]{%
  \def\@subjclass{#2}%
  \@ifundefined{subjclassname@#1}{%
    \ClassWarning{\@classname}{Unknown edition (#1) of Mathematics
      Subject Classification; using '1991'.}%
  }{%
    \@xp\let\@xp\subjclassname\csname subjclassname@#1\endcsname
  }%
}
 \makeatother

\newtheorem{theorem}{Theorem}[section]

\newtheorem{lemma}[theorem]{Lemma}
\newtheorem{corollary}[theorem]{Corollary}
\newtheorem{proposition}[theorem]{Proposition}
\theoremstyle{definition}

\newtheorem{remark}[theorem]{Remark}
\numberwithin{equation}{section}

\newcounter{minutes}
\divide\time by 60
\newcounter{hours}
\multiply\time by 60 \addtocounter{minutes}{-\time}
\begin{document}
\title{On J. C. C. Nitsche type inequality for annuli on Riemann surfaces}
\author{David Kalaj}
\address{
Faculty of Natural Sciences and Mathematics, University of
Montenegro, Cetinjski put b.b. 81000 Podgorica, Montenegro}
\email{davidk@ac.me}

\begin{abstract}
 Assume that $(\mathcal{N},\hbar)$ and $(\mathcal{M},\wp)$ are two
Riemann surfaces with conformal metrics $\hbar$ and $\wp$. We prove
that if there is a harmonic homeomorphism between an annulus
$\mathcal{A}\subset \mathcal{N}$ with a conformal modulus
$\mathrm{Mod}(\mathcal{A})$ and a geodesic annulus
$A_\wp(p,\rho_1,\rho_2)\subset \mathcal{M}$, then we have
${\rho_2}/{\rho_1}\ge \Psi_\wp\mathrm{Mod}(\mathcal{A})^2+1,$ where
$\Psi_\wp$ is a certain positive constant depending on the upper
bound of Gaussian curvature of the metric $\wp$. An application for
the minimal surfaces is given.
\end{abstract}
\maketitle

\section{Introduction}

\subsection{Background and statement of the main result}  In order to include some background related
to the problem treated in this paper assume that $\mathbf{A}(a,b)$
is the annulus $\{x\in\mathbf{R}^n:a<|x|<b\}$ in the Euclidean space
$\mathbf{R}^n$, $n\ge 2$.  Fifty years ago J. C. C. Nitsche
\cite{n}, studying the minimal surfaces and inspired by radial
harmonic homeomorphism between annuli, by using Harnack inequality
for positive harmonic functions in $\mathbf{R}^n$ proved that the
existence of an Euclidean harmonic homeomorphism between annuli
$\mathbf{A}(r,1)$ and $\mathbf{A}(\rho,1)$ is equivalent with an
inequality $\rho<\rho_r$, where $\rho_r$ is a positive constant
smaller than 1. Then he conjectured  that for $n=2$
\begin{equation}\label{nit}\rho_r=
\frac{2r}{1+r^2}.\end{equation} This conjecture is solved recently
in positive  by Iwaniec, Kovalev and Onninen in \cite{conj}.  Some
partial results have been obtained previously by Lyzzaik \cite{Al},
Weitsman \cite{aw} and the author \cite{israel}. On the other hand
in \cite{h} and in \cite{kalaj} is treated the same problem for the
harmonic mappings w.r.t the hyperbolic and Riemann metric in
two-dimensional hyperbolic space and in two-dimensional Riemann
sphere respectively.  The $n-$dimensional generalization of J. C. C.
Nitsche conjecture is
\begin{equation}\label{nitge}{\rho_r}=\frac{nr}{n-1+r^n}\end{equation} and is inspired by the radial
harmonic mapping
\begin{equation}\label{f}f(x)=\left(\frac{1-r^{n-1}\rho}{1-r^n}+\frac{r^{n-1}\rho-r^n}{(1-r^n)|x|^n}\right)x\end{equation}
between annuli $\mathbf{A}(r,1)$ and $\mathbf{A}(\rho,1)$ (c.f.
\cite{jmaa}). In \cite{kalaj,jmaa} the author treated the
three-dimensional case and obtained an inequality for Euclidean
harmonic mappings between annuli on $\mathbf{R}^3$.  The last
conjectured inequality for $n\ge 3$ remains an open problem.

The Nitsche phenomena is  rooted in the theory of doubly connected
minimal surfaces (cf. \cite{conj}) and in the existence problem of
 maps between annuli of the complex plane,
Riemann surface or Euclidean space with the least Dirichlet energy
(cf. \cite{io}, \cite{dist}, \cite{gaven}).

In this paper we consider  harmonic mappings between certain annuli
of the Riemannian surfaces and study J. C. C. Nitsche type problem
for the harmonic mappings between geodesic annuli of Riemannian
surfaces. By $(\mathcal{M},\wp)$ we denote a Riemann surface with
the conformal metric $\wp(w)|dw|$, whose Gaussian curvature
$K=K_\wp(w)$ is bounded from the above by a constant $\pm \kappa^2$
($\sup K_\wp(w)=\kappa^2$ or $\sup K_\wp(w)=-\kappa^2$). By
$B_\wp(p,\rho_p)$ we denote the geodesic ball with the center at
 $p\in \mathcal{M}$
 which is disjoint from the cut-locus of $p$ with
the radius $\rho_p>0$ (which for $\sup K=\kappa^2$ satisfies the
additional condition $\rho_p\le{\pi}/({2\kappa})$).
 The geodesic annulus is
defined by
$A_\wp(p,\rho_1,\rho_2)=B_\wp(p,\rho_2)\setminus\overline{B_\wp(p,\rho_1)}$.
(We refer to section~\ref{lopar} for more details concerning the
above concepts.) The following is the main result of this paper
\begin{theorem}[The main result]\label{nimat}
If there exists a harmonic homeomorphism between an annulus
$A(r_1,r_2)\subset \mathbf{R}^2$ and a geodesic annulus
$A_\wp(p,\rho_1,\rho_2)$  of the Riemann surface
$(\mathcal{M},\wp)$, then we have
\begin{equation}\label{new1}\frac{\rho_2}{\rho_1}\ge
\Psi_\wp\log^2\frac{r_2}{r_1}+1,\end{equation} where
\begin{equation}\label{10}\Psi_\wp=\left\{
                                                                       \begin{array}{ll}
                                                                         \frac{ \sinh\left[\kappa
\rho_1\right]}{2\kappa\rho_1}, & \hbox{if $\sup_{w\in M} K_\wp(w)= -\kappa^2$;} \\
                                                                         \frac{1}{2}, & \hbox{if $\sup_{w\in M} K_\wp(w)=  0$;} \\
                                                                         \frac{ \sin\left[\kappa \rho_1\right]}{2\kappa\rho_1}, & \hbox{if $\sup_{w\in M} K_\wp(w)= \kappa^2$.}
                                                                       \end{array}
                                                                     \right.\end{equation}

\end{theorem}It is well-known the following fact, a simply connected minimal
surface $\Sigma\subset \mathbf{R}^3$, lying over a simply connected
domain $\Omega\neq\mathbf{R}^2$, can be parameterized by a conformal
harmonic parametrization (Weierstrass-Enneper parametrization) $
w=f(z)=(f_1(z),f_2(z),f_2(z)):D(0,1)\to \Sigma$, where $D(0,1)$ is
the unit disk. Moreover $\Sigma$ is a Riemann surface with a
conformal metric $\wp=(|g'(z)|+|h'(z)|)|dz|$ with negative Gaussian
curvature. Here $g$ and $h$ are certain holomorphic functions in the
unit disk $D(0,1)$ such that $f(z)=g(z)+\overline{h(z)}$ is a
harmonic diffeomorphism of the unit disk onto the domain  $\Omega$
(cf. \cite[Chapter~10]{duren}).

The conformal modulus of an annulus $\mathcal{A}$ in complex plane
or in a Riemann surface is
$\mathrm{Mod}(\mathcal{A}):=\log\frac{r_2}{r_1}$, where $r_1$ and
$r_2$ are inner and outer radii of a circular annulus $A=A(r_1,r_2)$
conformally equivalent to $\mathcal{A}$.

By using Theorem~\ref{nimat} (i.e. its reformulated version
Theorem~\ref{nima})  and the previous facts we have
\begin{corollary}\label{pok} Assume that $(\Sigma,\wp)$ is a simply connected minimal surface other than a whole plane.
Then for every geodesic annulus $A_\wp(\tilde
0,\rho_1,\rho_2)\subset \Sigma$ there hold the inequality
\begin{equation}\label{newmal}\frac{\rho_2}{\rho_1}>
\frac{1}{2}\mathrm{Mod}(A_\wp(\tilde
0,\rho_1,\rho_2))^2+1,\end{equation} or what is the same for a
circular annulus $A=A(\rho_1,\rho_2)\subset \mathbf{R}^2$ and for
the exponential map $\exp=\exp_{\tilde 0}$ (cf.
subsection~\ref{ss}), we have
$$\exp(\mathrm{Mod}(A))>
\frac{1}{2}\mathrm{Mod}(\exp(A))^2+1.$$
\end{corollary}
\begin{proof}[Proof (of Corollary~\ref{pok})] Let $w=f(z)=(f_1(z),f_2(z),f_2(z)):D(0,1)\to \Sigma$ be a
Weierstrass-Enneper parametrization of a minimal surface
$\Sigma\subset \mathbf{R}^3$ such that $f(0)=\tilde 0$. Let $\hat
A\subset D(0,1)$ be the pull-back of $A_\wp(\tilde 0,\rho_1,\rho_2)$
under $f$. Since $f$ is conformal, we obtain that
$$\mathrm{Mod}(\hat A)=\mathrm{Mod}(A_\wp(\tilde 0,\rho_1,\rho_2)).$$
Since the Gaussian curvature is negative then the function $\Psi$
defined in \eqref{10} satisfies the inequality
$\Psi_\wp(K)\ge\Psi_\wp(0)= 1/2$. From Theorem~\ref{nima} we obtain
\eqref{newmal}. In order to obtain the second inequality we only
need to point out that $\exp(A(\rho_1,\rho_2))=A_\wp(\tilde
0,\rho_1,\rho_2)$ which follows from the elementary properties of
exponential map.
\end{proof}
\begin{remark}
In connection with Theorem~\ref{nimat} (Theorem~\ref{nima}), we
should notice that the ratio $\rho_2/\rho_1$ determines the modulus
of annulus $A_\wp(p,\rho_1,\rho_2)$ only in the surface with a flat
metric (a metric with zero Gaussian curvature). This means in
particular that we cannot obtain in general an inequality involving
only the conformal moduli of an annulus and of its harmonic image.
\end{remark}
Theorem~\ref{nimat} is an extension of the main results in
\cite{kalaj,israel,h,aw} where it is proved the same result but only
for the Euclidean, Riemann  and Hyperbolic metric on the unit disk.
It can be considered as a counterpart of some recent results of
Iwaniec, Koh, Kovalev and Oninnen \cite{inve,IKO4} where the authors
established the existence of harmonic diffeomorphisms between the
annuli in complex plane, and of the author \cite{exis1}, where the
author established the existence of harmonic diffeomorphisms between
annuli on Riemann surface, provided that the conformal modulus of
domain is less or equal to the conformal modulus of the target.  The
proof of Theorem~\ref{nimat} is given in section~\ref{selo} and it
depends on the geometry of Riemann surface, more precisely we use a
version of Hessian comparison theorem of Yau and Schoen, which is
one of most fundamental tools of Riemannian geometry.  We make use
as well of a comparison principle due to Osserman. Two of the key
steps of the proof are: computation of the Laplacian of the distance
function of a harmonic map in geodesic polar coordinates with
variable Gaussian curvature and comparison of that Laplacian with
the Laplacian of distance function of a harmonic mapping in geodesic
polar coordinates with constant Gaussian curvature.
\section{Preliminaries}
\subsection{Harmonic mappings between Riemann surfaces}\label{lopar} Let $(\mathcal{N},\hbar)$ and $(\mathcal{M},\wp)$ be
Riemann surfaces with conformal metrics $\hbar$ and $\wp$,
respectively. If a mapping
$f:(\mathcal{N},\hbar)\to(\mathcal{M},\wp)$ is $C^2$, then $f$ is
said to be harmonic  ($\wp$-harmonic) if
\begin{equation}\label{el}
f_{z\overline z}+{(\log \wp^2)}_w\circ f f_z\,f_{\bar z}=0,
\end{equation}
where $z$ and $w$ are the local parameters on $\mathcal{N}$ and
$\mathcal{M}$ respectively (see \cite{jost}). Also $f$ satisfies
\eqref{el} if and only if its Hopf differential
\begin{equation}\label{anal}
\Psi[f]=\wp^2 \circ f f_z\overline{f_{\bar z}}
\end{equation} is a
holomorphic quadratic differential on $\mathcal{N}$.

Let $D$ be a domain in $\Bbb C$ and $\wp$ be a conformal metric in
$D$. The Gaussian curvature of the smooth metric $\wp$ is given by
$$K_\wp=- \frac{\triangle\log \wp }{\wp^2 }.$$
\subsection{Geodesic polar coordinates}\label{ss} Assume
that $(\mathcal{M},\wp)$ is a Riemannian surface with conformal
metric $\wp(w)|dw|$ whose Gaussian curvature $K$ is bounded from
above by a constant $\sup K=\pm\kappa^2$. The inner product in the
tangent space $T_wM$ is given by
$g_w(\zeta,\xi):=\wp^2(w)\left<\zeta,\xi\right>$. The distance
function $d_\wp$ is defined as follows
$$d_\wp(p,q):=\inf_{p,q\in\gamma}\int_\gamma\wp(w)|dw|,$$ where the infimum runs over all rectifiable Jordan arcs $\gamma\subset \mathcal{M}$ connecting $p$ and
$q$. The disk on the Riemannian surface with the center $p\in
\mathcal{M}$ and the radius $\rho_p>0$ is defined as
$$B_\wp(p,\rho_p)=\{q\in \mathcal{M}: d_\wp(p,q)<\rho_p\}.$$ The
disc $B_\wp(p,\rho_p)$ is called geodesic, if for any $a,b\in
B_\wp(p,\rho_p)$ it exists a geodesic curve $c\subset
B_\wp(p,\rho_p)$ joining $a$ and $b$. Equivalently,
$B_\wp(p,\rho_p)$ is the diffeomorphic image of the Euclidean open
disk $D(0,\rho_p)$ under the exponential map $$\exp: D(0,\rho_p)\to
B_\wp(p,\rho_p)\subset \mathcal{M}.$$ Let
\begin{equation}\label{folo}ds^2=d\rho^2+G^2(\rho,\theta) d\theta^2\end{equation} be the metric in
geodesic polar coordinates on the geodesic disk $B_\wp(p,\rho_p)$,
$p\in \mathcal{M}$, where $\rho_p>0$ and for $\sup K=\kappa^2$ it
satisfies the additional condition
\begin{equation}\label{jostj}\rho_p\le\frac{\pi}{2\kappa}.\end{equation}
The condition \eqref{jostj} is related to the comparison theorem of
Morse-Schoenberg, which implies the local diffeomorphic behavior of
exponential map,  and the last fact is used by Jost in the proof of
\cite[Theorem~2.1]{jost} for the existence of a geodesic ball. The
geodesic annulus with the inner and outer radii $\rho_1$ and
$\rho_2$ is defined as
$$A_\wp(p,\rho_1,\rho_2)=\{q\in \mathcal{M}:
\rho_1<d_\wp(p,q)<\rho_2\},$$ where $\rho_2$ satisfies
\eqref{jostj}. We shall use the following well-known facts, which
follows from \eqref{folo}
\begin{equation}\label{prima}
g_w\left(\frac{\partial}{\partial\rho}\big|_w,\frac{\partial}{\partial\rho}\big|_w\right)=1
\end{equation}
 \begin{equation}\label{seconda}
g_w\left(\frac{\partial}{\partial\theta}\big|_w,\frac{\partial}{\partial\theta}\big|_w\right)=G^2(\rho,\theta)
\end{equation}

 \begin{equation}\label{terca}
g_w\left(\frac{\partial}{\partial\theta}\big|_w,\frac{\partial}{\partial\rho}\big|_w\right)=0.
\end{equation}

 For the definition of the above concepts and
detailed study of the Riemannian metrics we refer to
\cite[Chapter~1]{rga}.

\subsection{Rotationally symmetric metrics  and harmonic mappings}
The following lemma is essentially proved in the author's paper
\cite{kalaj}, but for the completeness and for the future reference
we include its simplified proof here.
\begin{lemma}\label{lobi} Assume that $\wp$ is a rotationally symmetric metric,
i.e. $\wp(w) = h(|w|)$ for some smooth real function $h$ defined in
a segment $[0,a]$. Let $f$ be a $\wp-$harmonic mapping of the unit
disk onto $D_\wp(p,s)$ and assume that $(\rho(z),\theta(z))$ are
geodesic polar coordinates of the point $w=f(z)$. Then
$$G(\rho,\theta)=\eta(\rho)/\eta'(\rho)$$ where $\eta^{-1}(|z|)=d_\wp(p,z)$ and \begin{equation}\label{mal}\triangle
\rho(z)= \frac{1}{2}\frac{\partial G^2}{\partial \rho} |\nabla
\theta|^2.\end{equation} \end{lemma}

\begin{proof}
Let $g$ be the inverse of the function $s\mapsto d_\wp(s,0)$. Then
we have
$$\rho = \int_0^{g(\rho)} h(t)dt.$$
Thus \begin{equation}\label{drita}1 = g'(\rho) \cdot
h(g(\rho)),\end{equation} and
\begin{equation}\label{kalaj}\frac{h'}{h} =
-\frac{g''}{{g'}^2}.\end{equation} Therefore the metric of the
surface can be expressed as
$$ds^2=\wp(w)^2|dw|^2 = h(\rho)^2|d(g(\rho)e^{i\theta})|^2=d\rho^2 + G^2(\rho) d\theta^2,$$ where
\begin{equation}\label{metro}G^2(\rho)=h^2(g(\rho))g^2(\rho).\end{equation}
If $w=f(z)$, is a twice differentiable, then
$$f(z)= g(\rho(z))e^{i\theta}.$$
Now we have
$$w_x = (g'\rho_x + i g\theta_x
)e^{i\theta},$$
$$w_y = (g'\rho_y + i g\theta_y)e^{i\theta},$$
and thus \begin{equation}\label{prima1}w_{xx} = (g''
\rho_x^2+g'\rho_{xx} +2ig'\rho_x\theta_x +i g\theta_{xx}
-g\theta_x^2)e^{i\theta},\end{equation}
\begin{equation}\label{seconda1}w_{yy} = (g'' \rho_y^2+g'\rho_{yy} +2ig'\rho_y\theta_y +i
g\theta_{yy} -g\theta_y^2)e^{i\theta},\end{equation} and
\begin{equation}\label{terca1}w_z w_{\bar z} = \frac{1}{4}\left(w_x^2 +
w_y^2\right).\end{equation} Assume now that $w$ is harmonic.  By
applying \eqref{prima1}, \eqref{seconda1}, \eqref{terca1} and
\eqref{el} in view of
\begin{equation}\label{ew}\partial_w \log
\wp^2(w)=\frac{h'(|w|)\overline{w}}{h(|w|)|w|},\end{equation} it
follows that
\[\begin{split}(g'' |\nabla\rho|^2&+g'\triangle\rho+2ig'\left<\nabla\rho,\nabla\theta\right>
+ig\triangle \theta -g|\nabla\theta|^2)e^{i\theta}
\\&+\frac{h'(g(\rho)) e^{-i\theta}}{h(g(\rho))}\left({g'}^2|\nabla\rho|^2 +
2ig'\left<\nabla\rho,\nabla\theta\right> -g^2|\nabla \theta|^2)
\right)e^{2i\theta} = 0.
\end{split}
\]
Thus
\[\begin{split}(g''
|\nabla\rho|^2&+g'\triangle\rho+2ig'\left<\nabla\rho,\nabla\theta\right>
+ig\triangle \theta -g|\nabla\theta|^2)\\&
+\frac{h'(g(\rho))}{h(g(\rho))}\left({g'}^2|\nabla\rho|^2 +
2ig'\left<\nabla\rho,\nabla\theta\right> -g^2|\nabla \theta|^2)
\right)=0.
\end{split}
\]
Therefore
\begin{equation}\label{imagine}2g'\left<\nabla\rho,\nabla\theta\right> +g\triangle \theta
+2\frac{h'(g(\rho))g'(\rho)}{h(g(\rho))}\left<\nabla\rho,\nabla\theta\right>=0\end{equation}
and \begin{equation}\label{reale}(g'' |\nabla\rho|^2+g'\triangle\rho
-g|\nabla\theta|^2)
+\frac{h'(g(\rho))}{h(g(\rho))}\left({g'}^2|\nabla\rho|^2
 -g^2|\nabla \theta|^2) \right) = 0.
\end{equation}
Combining \eqref{kalaj} and \eqref{imagine} it follows that
\begin{equation}\label{dk}g'\triangle\rho =\left(g(\rho)
+\frac{h'(g(\rho))}{h(g(\rho))}
 g^2\right)|\nabla \theta|^2.
\end{equation}
From \eqref{drita} we obtain
\begin{equation}\label{mal0}\triangle\rho =
g(\rho)\left[{h(g(\rho))}+h'(g(\rho))
 g(\rho)\right]|\nabla \theta|^2\end{equation} i.e.
 \begin{equation*}\triangle\rho =
\frac{1}{2}\frac{\partial G^2}{\partial \rho}|\nabla
\theta|^2.\end{equation*}
\end{proof}
The proof of the previous lemma contains in particular the following
lemma.
\begin{lemma}
For $\kappa\neq 0$, the Gaussian curvature of the metric
$$\hat\wp=\frac{2|dz|}{\kappa(1\pm|z|^2)}\text{ is }K(z)=\pm \kappa^2\neq 0,$$
while for $\kappa=0$, the corresponding metric is the Euclidean one:
$\hat\wp(w)=|dw|$. Its distance function is
$$r(|z|)=\left\{
           \begin{array}{ll}
             \frac{2}{\kappa}\tan^{-1}(|z|), & \hbox{if $K(z)=\kappa^2$;} \\
             \frac{2}{\kappa}\tanh^{-1}(|z|), & \hbox{if $K(z)=-\kappa^2$;}\\
             |z|, & \hbox{if $K(z)=0$.}
           \end{array}
         \right.
$$ Therefore
$$\hat g(\rho)=\left\{
                 \begin{array}{ll}
                   \tan\frac{2\rho}{\kappa}, & \hbox{if $K(z)=\kappa^2$;} \\
                   \tanh\frac{2\rho}{\kappa}, & \hbox{if $K(z)=-\kappa^2$;} \\
                   \rho, & \hbox{if $K(z)=0$.}
                 \end{array}
               \right.
$$ Thus for a constant curvature
surface we have
\begin{equation}\label{hat}\hat G(\rho,\theta)=\left\{
                                                              \begin{array}{ll}
                                                                \frac{\sin(\kappa\rho)}\kappa, & \hbox{if $K(z)=\kappa^2$;} \\
                                                                \rho, & \hbox{if $K(z)=0$} \\
                                                                \frac{\sinh(\kappa\rho)}\kappa, & \hbox{if $K(z)=-\kappa^2$}
                                                              \end{array}
                                                            \right..\end{equation}

\end{lemma}
 The
following function is well-known in the Riemann geometry, especially
in connection with Hessian and Laplacian comparison theorem.
$$h_c(r)=\left\{
                                 \begin{array}{ll}
                                   \sqrt{c}\cot(\sqrt{c} r), & \hbox{if $c>0$;} \\
                                   \frac{1}{r}, & \hbox{if $c=0$;} \\
                                   \sqrt{-c}\coth(\sqrt{-c} r), & \hbox{if $c<0$.}
                                 \end{array}
                               \right.$$
 The main result lies on the
inequality \eqref{MAL}, which will be proved by using the following
Hessian comparison theorem due to Yau and Schoen \cite{sy1}.

\begin{proposition}[Hessian Comparison Theorem](see \cite{sy1} cf. \cite[Theorem~3.2]{toby}). Let $\mathcal{M}$ be a Riemannian
manifold and $p$, $q\in \mathcal{M}$ be such that there is a
minimizing unit speed geodesic $\gamma$ joining $p$ and $q$, and let
$r(x) = \mathrm{dist}(p, x)$ be the distance function to $p$. Let
$K_\gamma\le c$
 be the radial sectional curvatures of $\mathcal{M}$ along $\gamma$. If
$c
> 0$ assume $r(q) < \frac{\pi}{2\sqrt{c}}$. Then we have
\begin{equation}\label{para}\mathrm{Hess}\, r(x)(\gamma', \gamma'  ) = 0 \end{equation} and
\begin{equation}\label{dyta}\mathrm{Hess}\,r(x)(X, X) \ge
h_c(r(x))|X|_g^2,\end{equation} where $X \in T_xM$ is perpendicular
to $\gamma'(r(x))$. Here $\mathrm{Hess}(r(x))$ is the Hessian matrix
of $r(x)$.

\end{proposition}

Now we formulate a comparison theorem of Osserman \cite{oser}.
\begin{theorem}[Comparison Theorem] Let $ d s^2 $ and $ d \hat{s}^2 $ be
metrics given in geodesic polar coordinates by
$$
d s^2 = d\rho^2 + G^2(\rho, \theta)^2 d \theta^2
$$
$$
d \hat{s}^2 = d \rho^2 + \hat{G}^2(\rho, \theta)^2 d \theta^2.
$$
If the Gaussian curvatures satisfies
\begin{equation}
\label{25} K(\rho, \theta) \leq \hat{K} (\rho, \theta), \quad 0 <
\rho < \rho_0,
\end{equation}
then
\begin{equation}
\label{26} \frac{1}{G^2} \, \frac{\partial G^2}{\partial \rho} \geq
\frac{1}{\hat{G}^2} \, \frac{\partial \hat{G^2}}{\partial \rho}
\end{equation}
and
\begin{equation}
\label{27} G^2(\rho, \theta) \geq \hat{G}^2 (\rho, \theta), \quad 0
< \rho < \rho_0.
\end{equation}
\end{theorem}
\section{The key lemmas}
\begin{lemma}\label{moti}
Let $w=f(z)$ be a harmonic mapping of an open subset $\Omega\subset
\mathbf{R}^2$ into the geodesic disk $B_g(\tilde 0,p)\subset
(\mathcal{M},\wp)$ of the Reimannian surface $(\mathcal{M},g)$ with
a Gaussian curvature $K\le c$, $\tilde 0\in \mathcal{M}$. Assume
that $ds^2=dr^2+G^2(r,\theta)d\theta^2$ is the metric in geodesic
polar coordinates in tangential space $T_{\tilde 0}\mathcal{M}$ of
the surface $\mathcal{M}$. Let $r$ be the distance function from the
fixed point $\tilde 0\in \mathcal{M}$. Define $\rho(z)=r(f(z))$.
Then we have the following inequality
\begin{equation}\label{MAL}\triangle\rho \ge h_c(\rho){ G^2(\rho,\theta)}|\nabla
\theta|^2,\end{equation} where $\triangle$ and $\nabla$ are standard
Euclidean Laplacian and Euclidean gradient respectively.
\end{lemma}
\begin{proof}
In order to obtain Laplacian of $\rho(z)$ we use an approach similar
to that on the book of Schoen and Yau (\cite[P.~176]{sy})
(cf.\cite[Eq.~5.1.2]{jost}).

Viewing $\rho$ as a composite function and applying the chain rule,
by using the notation $z=(z^1,z^2)=x+i y$ and assuming that
$w=(w^1,w^2)$  are Riemann normal coordinates on $\mathcal{M}$, we
get
\[\begin{split}\triangle d(w(z),\tilde 0)&=\sum_{i,\alpha} \left (\rho_i \frac{\partial
w^i}{\partial z^\alpha}\right)_\alpha
\\&=\sum_{i,j,\alpha}\rho_{ij}\frac{\partial w^i}{\partial
z^\alpha}\frac{\partial w^j}{\partial z^\alpha}+\sum_i
\rho_i\triangle w^i.\end{split}\] Here the lower indices denote the
partial derivatives. Since $w$ is harmonic, and $w^1$, $w^2$ are
normal coordinates, then the second term vanishes.
 Thus we obtain the formula
\begin{equation}\triangle d(w(z),\tilde 0)=\sum_{\alpha\in\{x,y\}}
\mathrm{Hess}(\rho)(\nabla_\alpha w,\nabla_\alpha w).\end{equation}
Moreover in geodesic polar coordinates we have $w=\rho e^{i\theta}$
and therefore
$$\nabla_x w=\nabla_x(\rho e^{i \theta})=\rho_x e^{i \theta}-i \rho
\theta_x e^{i\theta}$$ and
$$\nabla_y w=\nabla_y(\rho e^{i \theta})=\rho_y e^{i \theta}-i \rho
\theta_y e^{i\theta}.$$ The last two relations imply
\begin{equation}\label{cile}\triangle d(w(z),\tilde
0)=|\nabla\rho|^2\mathrm{Hess}(\rho)(e^{i\theta},e^{i\theta})+|\nabla\theta|^2\mathrm{Hess}(\rho)(\rho
ie^{i\theta},\rho ie^{i\theta}).\end{equation} Since Gaussian
curvature $K\le\pm\kappa^2$, then the radial sectional curvature
along $\gamma(t)=\exp_{\tilde 0}(te^{i\theta})$ is also bounded by
$\pm\kappa^2$. For $w=f(z)$ we have
\begin{equation}\label{cile1}\frac{\partial}{\partial\theta}\big|_w = \rho ie^{i\theta}\end{equation}  and
\begin{equation}\label{cile2}\left|\frac{\partial}{\partial\theta}\big|_w\right|^2=
\left<\frac{\partial}{\partial\theta}\big|_w,\frac{\partial}{\partial\theta}\big|_w\right>=
g(\frac{\partial}{\partial\theta}\big|_w,\frac{\partial}{\partial\theta}
\big|_w) = G^2(\rho,\theta).\end{equation}   Because of \eqref{para}
and
\begin{equation}\label{cile3}\gamma'(\rho(z))=\nabla\rho= e^{i\theta},\end{equation}   we obtain
\begin{equation}\label{cile4}|\nabla\rho|^2\mathrm{Hess}(\rho)(e^{i\theta},e^{i\theta})=0.\end{equation}  By using now \eqref{dyta} and \eqref{cile} -- \eqref{cile4}
we obtain
\[\begin{split}\triangle\rho&=\triangle d(w(z),\tilde
0)\\&=|\nabla\theta|^2\mathrm{Hess}(\rho)(\rho ie^{i\theta},\rho
ie^{i\theta})\\&=|\nabla\theta|^2\mathrm{Hess}(\rho)\left(\frac{\partial}{\partial\theta}\big|_w
,\frac{\partial}{\partial\theta}\big|_w\right)\\&\ge h_\kappa(\rho)
|\nabla\theta|^2{\left|\frac{\partial}{\partial\theta}\big|_w\right|^2}\\&=
h_\kappa(\rho)G^2(\rho,\theta)|\nabla \theta|^2,
\end{split}\] which yields \eqref{MAL}.
\end{proof}
By using comparison theorem of Osserman we obtain
\begin{lemma} Under the condition of Lemma~\ref{moti}, for the Rimannian surfaces $(\mathcal{M},\wp)$ with Gaussian curvature $K$ bounded
from above we have the following sharp inequality
\begin{equation}\label{negat}\triangle\rho\ge  \psi_\wp(\rho)|\nabla \theta|^2,\end{equation} where
\begin{equation}\label{psiw}\psi_\wp(\rho)=\left\{
                                                                       \begin{array}{ll}
                                                                         \frac{ \sinh\left[\kappa
\rho\right]}{\kappa}, & \hbox{if $\sup K= -\kappa^2$;} \\
                                                                         {\rho}, & \hbox{if $\sup K=  0$;} \\
                                                                         \frac{ \sin\left[\kappa \rho\right]}{\kappa}, & \hbox{if $\sup K= \kappa^2$.}
                                                                       \end{array}
                                                                     \right.\end{equation}
\end{lemma}
\begin{proof}
 Let $ d s^2=d\rho^2+G^2(\rho,\theta)d\theta^2 $ and $ d \hat{s}^2 =d\rho^2+\hat{G}^2(\rho,\theta)d\theta^2 $ be
metrics given in geodesic polar coordinates that present the metric
$\wp$  and the constant curvature metric $\hat\wp$ ($K\equiv
\pm\kappa^2$). Combining \eqref{MAL}, \eqref{27} and \eqref{hat} we
arrive at \eqref{negat}.  From Lemma~\ref{lobi}, we conclude that
the inequality \eqref{negat} reduces to an equality for constant
curvature metrics. This proves the sharpness of \eqref{negat}.
\end{proof}

\section{The proof of main result}\label{selo}  We are now prepare to prove
Theorem~\ref{nimat}, i.e. its slight reformulation:
\begin{theorem}\label{nima} Assume that $\mathcal{N}$ and $\mathcal{M}$ are Riemann
surfaces and let $\mathcal{A}$ be a doubly connected domain in
$\mathcal{N}$. If there exists a harmonic homeomorphism $f$ between
annuli $\mathcal{A}\subset \mathcal{N}$ and
$A_\wp(p,\rho_1,\rho_2)\subset \mathcal{M}$, then we have
\begin{equation}\label{new}\frac{\rho_2}{\rho_1}\ge
\Psi_\wp\mathrm{Mod}(\mathcal{A})^2+1,\end{equation} where
$$\Psi_\wp=\left\{
                                                                       \begin{array}{ll}
                                                                         \frac{ \sinh\left[\kappa
\rho_1\right]}{2\kappa\rho_1}, & \hbox{if $\sup K= -\kappa^2$;} \\
                                                                         \frac{1}{2}, & \hbox{if $\sup K=  0$;} \\
                                                                         \frac{ \sin\left[\kappa \rho_1\right]}{2\kappa\rho_1}, & \hbox{if $\sup K= \kappa^2$.}
                                                                       \end{array}
                                                                     \right.$$
Recall that the case $\sup K= \kappa^2$, is subject of the a priory
condition \eqref{jostj} for $\rho_2$ i.e. of
$\rho_2\le{\pi}/{(2\kappa)}$.
\end{theorem}

In order to prove Theorem~\ref{nima} we make use of the following
proposition.
\begin{proposition} Let
$w=\rho e^{i\theta}$ be a $C^1$ surjection between the rings
$A(r_1,r_2)$ and $A(s_1,s_2)$ of the complex plane
$\mathbf{C}=\mathbf{R}^2$. Then
\begin{equation}\label{zero}
\int_{r_1 \leq|z|\leq r_2 }|\nabla \theta
|^2\,\mathrm{d}x\,\mathrm{d}y\geq 2\pi\log\frac{r_2}{r_1}.
\end{equation}
\end{proposition}

For its proof see for example \cite{israel}.
\begin{proof}[Proof of Theorem~\ref{nima}] Assume for a moment that $f$
is smooth up to the boundary and $\mathcal{A}$ is the circular
annulus $A(r_1,r_2)$ and $f$ maps the inner boundary onto the inner
boundary.
 By applying Green's formula for $\rho$ on
$\{z:r_1 \leq|z|\leq \sigma \}$, $r_1<\sigma<r_2$, we obtain
\begin{equation*}
\int_{|z|=\sigma}\frac{\partial \rho(z)}{\partial \sigma}\,|dz|-
\int_{|z|=r_1}\frac{\partial \rho(z)}{\partial \sigma}\,|dz|
=\int_{r_1\leq|z|\leq \sigma }\triangle \rho \, d\mu.
\end{equation*}
Here $d\mu$ is the twodimensional Lebesgue measure. Since
$\frac{\partial \rho}{\partial \sigma}\ge 0$ for $|z|=r_1$ we obtain
$$\int_{|z|=\sigma}\frac{\partial \rho(z)}{\partial \sigma}\,|dz|\geq \int_{r_1 \leq|z|\leq \sigma }\triangle \rho \, d\mu.
$$  By applying (\ref{negat}) and
(\ref{zero}), and using the fact that the function $\psi_\wp$
defined in \eqref{psiw} is an increasing function in $\rho$, we
obtain
\begin{equation*}
\begin{split}
\int_{|z|=\sigma}\frac{\partial \rho(z)}{\partial \sigma}\,|dz|
&\geq \int_{r_1\leq|z|\leq \sigma }\triangle \rho \,
d\mu\\&\ge\int_{r_1 \leq
|z|\leq \sigma }\psi_\wp(\rho)|\nabla \theta|^2 d\mu \\
&\geq \psi_\wp(\rho_1)\int_{r_1\leq|z|\leq \sigma }|\nabla
\theta|^2d\mu\\&\geq 2\pi \psi_\wp(\rho_1)\log\frac{\sigma}{r_1}.
\end{split}
\end{equation*}
 It follows that
\begin{equation*}
\sigma\frac{\partial }{\partial \sigma}
\int_{|\zeta|=1}\rho(\sigma\zeta) \,|d\zeta|\geq 2\pi
\psi_\wp(\rho_1)\log\frac{\sigma}{r_1}.
\end{equation*} Dividing by
$\sigma$ and integrating over $[r_1,r_2]$ by $\sigma$ the previous
inequality, we get \[
\begin{split}
\int_{|\zeta|=1}\rho(r_2\zeta)\,|d\zeta|&-\int_{|\zeta|=1}
\rho(r_1\zeta)\,|d\zeta|\\&\geq  \pi \psi_\wp(\rho_1)\log^2
{\frac{r_2}{r_1}}
\end{split}
\]
i.e.
\begin{equation}\label{similar}
\begin{split}
2\pi(\rho_2 - \rho_1)\geq \pi \psi_\wp(\rho_1)\log^2
\frac{r_2}{r_1}.
\end{split}
\end{equation}
If $f$ is not smooth up to the boundary, then instead of $f$ we
consider the mapping $f_\epsilon$, $0<\epsilon<(\rho_2-\rho_1)/2$
constructed as follows. Let
$A_\wp(p,\rho_1+\epsilon,\rho_2-\epsilon)\subset
A_\wp(p,\rho_1,\rho_2)$ and define
$A_\epsilon=f^{-1}(A_\wp(p,\rho_1+\epsilon,\rho_2-\epsilon))$. Take
a conformal mapping $\phi_\epsilon: A(r_\epsilon,R_\epsilon)\to
A_\epsilon$ and define $f_\epsilon=f\circ \phi_\epsilon$. Then
$f_\epsilon:A_\epsilon\to A_\wp(p,\rho_1+\epsilon,\rho_2-\epsilon) $
is a $\wp$-harmonic diffeomorphism because its Hopf differential
$$\Psi[f_\epsilon]=\Psi[f](\phi_\epsilon)(\psi'_\epsilon(z))^2,$$ is
holomorphic. Moreover $$\lim_{\epsilon\to 0 }\log
\frac{R_\epsilon}{r_\epsilon}=\mathrm{Mod}(\mathcal{A}).$$ Then we
apply the previous case and let $\epsilon\to 0$ in order to obtain
\eqref{new} if $f$ maps the inner boundary onto the inner boundary.
If $f$ maps the inner boundary onto the outer boundary, then we
consider the composition of $f$ by the conformal mapping
$\phi(z)=r_1r_2/z$.
\end{proof}

\end{document}